\documentclass{svproc}
\usepackage{amsmath,amssymb,amsfonts}
\usepackage[mathscr]{eucal}
\usepackage{url}

\DeclareSymbolFont{rsfscript}{OMS}{rsfs}{m}{n}
\DeclareSymbolFontAlphabet{\mathrsfs}{rsfscript}

\DeclareMathOperator{\Aut}{Aut}

\begin{document}
\mainmatter              
\title{Block-groups and Hall relations\thanks{Supported by the Ministry of Science and Higher Education of the Russian Federation (Ural Mathematical Center project No. 075-02-2020-1537/1)}}
\titlerunning{Block-groups and Hall relations}
\author{Azza M. Gaysin\inst{1} \and Mikhail V. Volkov\inst{2}}
\authorrunning{Azza M. Gaysin, Mikhail V. Volkov} 
\tocauthor{Azza M. Gaysin, Mikhail V. Volkov}
\institute{Department of Algebra, Faculty of Mathematics and Physics, Charles University, Sokolovska 83, 186 00 Praha 8, Czech Republic\\
\email{azza.gaysin@gmail.com}
\and
Chair of Algebra and Theoretical Computer Science, Institute of Natural Sciences and Mathematics, Ural Federal University, Lenina 51, Ekaterinburg 620000, Russia\\
\email{m.v.volkov@urfu.ru},\\
WWW home page: \url{http://csseminar.kmath.ru/volkov/}
}

\maketitle              

\begin{abstract}
A binary relation on a finite set is called a Hall relation if it contains a permutation of the set. Under the usual relational product, Hall relations form a semigroup which is known to be a block-group, that is,
a semigroup with at most one idempotent in each $\mathrsfs{R}$-class and each $\mathrsfs{L}$-class. Here we show that in a certain sense, the converse is true: every finite block-group divides a semigroup of Hall relations on a finite set.
\keywords{Hall relation, reflexive relation, $\mathrsfs{J}$-trivial semigroup, block-group, power semigroup, semidirect product, semigroup division}
\end{abstract}
\section{Background and Motivation: Straubing's Theorem}
\label{sec:intro}
The result that we are going to present is inspired by Straubing's representation theorem for $\mathrsfs{J}$-trivial monoids~\cite{Straubing:1980}. Straubing's theorem involves three notions: $\mathrsfs{J}$-trivial semigroups, monoids of reflexive relations, and semigroup division. For the reader's convenience, we recall their definitions.

Given a semigroup $S$, we denote by $S^1$ the least monoid containing $S$, that is, $S^1:=S$ if $S$ has an identity element and  $S^1:=S\cup\{1\}$ if $S$ has no identity element; in the latter case the multiplication in $S$ is extended to $S^1$ in a unique way such that the fresh symbol $1$ becomes the identity element in~$S^1$. Green \cite{Green:1951} defined five important equivalencies on every semigroup $S$, collectively referred to as \emph{Green's relations}, of which we meet the following three in this note:
\begin{itemize}
\item[] $x\,\mathrsfs{R}\,y \Leftrightarrow xS^1 = yS^1$, i.e., $x$ and $y$ generate the same right ideal;
\item[] $x\,\mathrsfs{L}\,y \Leftrightarrow S^1x = S^1y$, i.e., $x$ and $y$ generate the same left ideal;
\item[] $x\,\mathrsfs{J}\,y \Leftrightarrow S^1xS^1 = S^1yS^1$, i.e., $x$ and $y$ generate the same ideal.
\end{itemize}
Basic information about $\mathrsfs{R}$, $\mathrsfs{L}$, and  $\mathrsfs{J}$ can be found in the early chapters of any general semigroup theory text such as, e.g., \cite{Clifford&Preston:1961,Howie:1995}, but actually this note uses only the above definitions of these three relations.

A semigroup $S$ is said to be $\mathrsfs{J}$-\emph{triv\-i\-al} if the relation $\mathrsfs{J}$ on $S$ coincides with the equality relation $\Delta_S$ on $S$. In other words, this means that the following implication holds for all $x,y\in S$:
\[
S^1xS^1=S^1yS^1\to x=y.
\]

Let $X$ be a set. Recall that binary relations on $X$ are multiplied as follows: for $\rho,\sigma\subseteq X\times X$, their product is set to be the relation
\[
\rho\sigma:=\{(x,y)\in X\times X\mid \exists z\in X\ (x,z)\in\rho\mathrel{\&}(z,y)\in\sigma\}.
\]
This multiplication is associative and $\Delta_X$, the equality relation on $X$, serves as the identity element for it. Thus, the binary relations on $X$ constitute a monoid. Also, it is easy to check that the multiplication is compatible with inclusions between relations: if $\rho\subseteq\rho'$ and $\sigma\subseteq\sigma'$, then $\rho\sigma\subseteq\rho'\sigma'$.

A binary relation $\rho$ on $X$ is \emph{reflexive} if $\rho$ contains $\Delta_X$. The above observations immediately imply that the reflexive relations on $X$ form a submonoid in the monoid of all binary relations on $X$.
Let $\mathscr{R}_n$ denote the monoid of all reflexive binary relations on a set with $n$ elements. This monoid can be conveniently thought of as a submonoid of the monoid of all $n\times n$ matrices (with the usual
matrix multiplication) over the Boolean semiring $\{0,1\}$, with the operations $+$ and $\cdot$ on $\{0,1\}$ being defined by the rules:
$$0\cdot 0=0\cdot 1=1\cdot 0=0+0=0,\qquad 1\cdot 1=1+0=0+1=1+1=1.$$
Namely, $\mathscr{R}_n$ can be identified with the submonoid consisting of matrices in which all diagonal entries are~1.

Finally, a semigroup $S$ is said to \emph{divide} another semigroup $T$ if $S$ is a homomorphic image of a subsemigroup in $T$. Now we are in a position to formulate Straubing's theorem.

\begin{theorem}[Straubing~\cite{Straubing:1980}]\label{thm:straubing}
A finite semigroup $S$ is $\mathrsfs{J}$-trivial if and only if $S$ divides the monoid $\mathscr{R}_n$ for some $n$.
\end{theorem}

\begin{remark}
In~\cite{Straubing:1980}, the above result is stated for $S$ being a monoid; this makes no essential difference since $S$ is $\mathrsfs{J}$-trivial if and only if so is $S^1$.
\end{remark}

Theorem~\ref{thm:straubing} looks quite innocent as it is stated in purely semigroup-theoretic terms and very much resembles textbook representation results such as the Cayley-type representation of arbitrary semigroups by transformations or binary relations. However, no direct semigroup-theoretic proof of Theorem~\ref{thm:straubing} is known. The proof in~\cite{Straubing:1980} crucially depends on Simon's theorem~\cite{Simon:1972,Simon:1975}, a deep combinatorial result in the theory of recognizable languages. Moreover, it can be shown relatively easily that Theorem~\ref{thm:straubing} and Simon's theorem are equivalent, and therefore, a~direct proof of the former would provide a new algebraic proof of the latter. In the literature, there are many proofs of Simon's theorem, based of different techniques, but none of the proofs are purely algebraic.

In the present note, we provide a representation by binary relations for another, larger class of finite semigroups, namely, the class of all finite block-groups. We introduce and briefly discuss this class in Section~\ref{sec:block-groups}, while in Section~\ref{sec:hall} we present a family $\{\mathscr{H}_n\}$ of monoids of binary relations consisting of so-called Hall relations. Our main result, which is stated and proved in Section~\ref{sec:main}, shows that the family $\{\mathscr{H}_n\}$ plays for block-groups precisely the same role as the family $\{\mathscr{R}_n\}$ plays for $\mathrsfs{J}$-trivial semigroups.

\section{Block-Groups and Power Semigroups of Groups}
\label{sec:block-groups}

Recall that an element $e$ of a semigroup $S$ is said to be an \emph{idempotent} if $e^2=e$. A \emph{block-group} is a finite semigroup with at most one idempotent in each $\mathrsfs{R}$-class and each $\mathrsfs{L}$-class.
This definition can be expressed by the following implications:
\begin{gather}
ef=e^2=e\mathrel{\&}fe=f^2=f\to e=f,\label{eq:bg1}\\
ef=f^2=f\mathrel{\&}fe=e^2=e\to e=f.\label{eq:bg2}
\end{gather}
Indeed, \eqref{eq:bg1} and respectively \eqref{eq:bg2} express the facts that any $\mathrsfs{R}$-related (respectively, $\mathrsfs{L}$-related) idempotents coincide.

We refer the reader to Pin's enthusiastic survey~\cite{Pin:1995} for an explanation of the name ``block-group''. The survey presents also remarkable and profound connections between block-groups and the theory of recognizable languages, especially its topological aspects. An unexpected connection of block-groups to computational complexity theory has been established in~\cite{Bulatov:2002}.

While the ``external'' connections just mentioned are of definite importance and interest, the present note is entirely ``internal'' with respect to the algebraic theory of block-groups. We need two results of this theory. The first one, due to Margolis and Pin~\cite{Margolis&Pin:1984}, relates block-groups to $\mathrsfs{J}$-trivial semigroups.
\begin{proposition}[\cite{Margolis&Pin:1984}, Proposition 2.3]\label{prop:margolis&pin}
A finite semigroup $S$ is a block-group if and only if the idempotents of $S$ generate a $\mathrsfs{J}$-trivial subsemigroup in $S$.
\end{proposition}

\begin{remark}
In~\cite{Margolis&Pin:1984}, the above result is stated for $S$ being a monoid; this makes no essential difference since $S$ is a block-group if and only if so is $S^1$. The same remark applies also to Theorem~\ref{thm:bg=pg} below, which also was originally stated for the case of monoids.
\end{remark}

Given a semigroup $S$, we denote by $\mathcal{P}(S)$ the set of all its non-empty subsets. One introduces an associative  multiplication on $\mathcal{P}(S)$ as follows: the product of subsets $A,B\in\mathcal{P}(S)$ is the subset
\[
AB:=\{ab\mid a\in A,\ b\in B\}.
\]
Then $\mathcal{P}(S)$ becomes a semigroup which is called the \emph{power semigroup of $S$}.

A jewel of the theory of block-groups is their characterization in terms of power semigroups of groups. This deep and difficult result is due to Henckell and Rhodes~\cite{Henckell:1991:BG=PG}, see also~\cite{Henckell:1991:typeII} for a detailed explanation and \cite{Steinberg:2000,Auinger&Steinberg:2005} for modern and shorter (but still complicated) proofs.

\begin{theorem}
\label{thm:bg=pg}
A finite semigroup $S$ is a block-group if and only if $S$ divides the power semigroup of some finite group.
\end{theorem}

Here we make a comment similar to that made after Theorem~\ref{thm:straubing}: even though the formulation of Theorem~\ref{thm:bg=pg} is purely semigroup-theoretic, all its proofs in the literature employ tools from outside algebra.

\section{Hall relations}
\label{sec:hall}
\setcounter{footnote}{0}

A binary relation $\rho\subseteq X\times X$ on a finite set $X$ is called a \emph{Hall relation} if $\rho$ contains a permutation of $X$. Here we treat permutations as binary relations, that is, given a permutation $\pi\colon X\to X$, we identify it with the relation $\{(x,x\pi)\mid x\in X\}$.

The name ``Hall relation'' was coined by Schwarz~\cite{Schwarz:1972,Schwarz:1973} with the reference to the classic marriage theorem by Hall~\cite{Hall:1935}. Indeed, Hall's theorem deals with perfect matchings in bipartite graphs, and if one represents binary relations on a finite set as bipartite graphs, Hall relations are precisely those whose graphs admit a perfect matching. In the representation of binary relations as matrices over the Boolean semiring, Hall relations correspond to matrices with permanent~1.

The product of two permutations considered as relations on $X$ coincides with their usual product in the group of all permutations on $X$. If $\rho,\rho'$ are Hall relations and $\pi,\pi'$ are permutations such that $\pi\subseteq\rho$ and $\pi'\subseteq\rho'$, the product $\rho\rho'$ contains the permutation $\pi\pi'$ whence $\rho\rho'$ is a Hall relation again. Since $\Delta_X$, the equality relation on $X$, is a Hall relation, the Hall relations on $X$ form a submonoid in the monoid of all binary relations on $X$. Let $\mathscr{H}_n$ denote the monoid of all Hall relations on the set $X_n:=\{1,2,\dots,n\}$. Clearly, $\mathscr{H}_n$ contains both the monoid $\mathscr{R}_n$ of all reflexive relations on $X_n$ and the group $\mathbb{S}_n$ of all permutations on $X_n$.

The monoid $\mathscr{H}_n$ turns out to be a block-group. This property of $\mathscr{H}_n$ can be extracted from results announced by Ki Hang Kim~\cite{Butler:1974}\footnote{This paper was published under the name Kim Ki-hang Butler; see the biography of Ki Hang Kim~\cite{Boyle:2013} for an explanation.}. The argument outlined in~\cite{Butler:1974} is of counting nature: the author exhibits a recursive formula for the number of idempotents in $\mathscr{H}_n$~\cite[Theorem~15]{Butler:1974}, and then he claims that the number coincides with both the number of $\mathrsfs{L}$-classes that contain idempotents and the number of $\mathrsfs{R}$-classes that contain idempotents~\cite[Corollary~17]{Butler:1974}. The research announcement~\cite{Butler:1974} contained no proofs, nor we found any proofs of claims made therein in later publications that dealt with monoids of Hall relations. Therefore, we include here a short counting-free argument.

\begin{proposition}\label{prop:kim}
The monoid $\mathscr{H}_n$  is a block-group.
\end{proposition}

\begin{proof}
Let $\rho$ be an idempotent from $\mathscr{H}_n$ and $\pi$ a permutation contained in $\rho$. There exists a positive integer $k$ such that $\pi^k=\Delta_{X_n}$. Since $\rho^2=\rho$, we have $\rho=\rho^k\supseteq\pi^k$, whence  $\rho\supseteq\Delta_{X_n}$. Thus, $\rho$ is reflexive, and we have shown that every idempotent of $\mathscr{H}_n$ lies in $\mathscr{R}_n$. The latter monoid is $\mathrsfs{J}$-trivial, and hence, $\mathscr{H}_n$  is a block-group by Proposition~\ref{prop:margolis&pin}.\qed
\end{proof}

\section{Representation Theorem}
\label{sec:main}

\begin{theorem}\label{thm:main}
A finite semigroup $S$ is a block-group if and only if $S$ divides the monoid $\mathscr{H}_n$ for some $n$.
\end{theorem}

\begin{proof}
The class of all block-groups is known to be closed under division (see, e.g., \cite{Margolis&Pin:1984}). Therefore the ``if'' part immediately follows from Proposition~\ref{prop:kim}.

For the ``only if'' part, we employ Theorem~\ref{thm:bg=pg}. Choose a finite group $G$ such that $S$ divides $\mathcal{P}(G)$ and let $n:=|G|$. It is sufficient to show that the semigroup $\mathcal{P}(G)$ embeds into the monoid $\mathscr{H}_n$. In order to simplify notation, we identify $G$ and $X_n$ as sets. Now, for each non-empty subset $A\in\mathcal{P}(G)$, define a binary relation $\rho_A$ as follows:
\[
\rho_A:=\{(g,h)\in G\times G\mid g^{-1}h\in A\}.
\]
Fix an element $a\in A$. By the definition, $\rho_A$ contains all pairs $(g,ga)$, where $g$ runs over $G$. As the relation $\{(g,ga)\mid g\in G\}$ is a permutation of $G$, we see that $\rho_A$ is a Hall relation. Thus, the map $f\colon A\mapsto\rho_A$ sends $\mathcal{P}(G)$ into $\mathscr{H}_n$.

We aim to show that $f\colon\mathcal{P}(G)\to\mathscr{H}_n$ is an embedding of semigroups. To see that $f$ is one-to-one, take any two different subsets $A,B\in\mathcal{P}(G)$. Without any loss, we may assume that $A\nsubseteq B$. If $a\in A\setminus B$, the pair $(e,a)$, where $e$ is the identity element of the group $G$, belongs to $\rho_A$ but not to $\rho_B$. Thus, $\rho_A\ne\rho_B$.

It remains to verify that $f$ is a homomorphism, that is, $\rho_A\rho_B=\rho_{AB}$ for arbitrary subsets $A,B\in\mathcal{P}(G)$. If $(x,y)\in\rho_A\rho_B$, there must exist an element $z$ such that $(x,z)\in\rho_A$ and $(z,y)\in\rho_B$. By the definition, we have $x^{-1}z\in A$ and $z^{-1}y\in B$, whence $x^{-1}y=x^{-1}z\cdot z^{-1}y\in AB$. We see that $(x,y)\in\rho_{AB}$. Thus, $\rho_A\rho_B\subseteq\rho_{AB}$.

To prove the opposite inclusion, take $(g,h)\in\rho_{AB}$. Then $g^{-1}h\in AB$, that is, $g^{-1}h=ab$ for some $a\in A$ and $b\in B$. We see that $g^{-1}hb^{-1}=a$, whence $(g,hb^{-1})\in\rho_A$, while $(hb^{-1})^{-1}h=b$, whence $(hb^{-1},h)\in\rho_B$. Therefore, we get $(g,h)\in\rho_A\rho_B$, as required.\qed
\end{proof}

As the above proof shows, Theorem~\ref{thm:main} is rather a straightforward consequence of Henckell and Rhodes's theorem (Theorem~\ref{thm:bg=pg}). After the formulation of Straubing's theorem (Theorem~\ref{thm:straubing}) in Section~\ref{sec:intro}, we said that it is more than a consequence of Simon's theorem: Theorem~\ref{thm:straubing} is in fact equivalent to Simon's theorem whence a direct algebraic proof of the former would provide a new algebraic proof of the latter. Could the same be said about the relationship between Theorem~\ref{thm:main} and Henckell and Rhodes's theorem?

To address this question, we need the concept of the semidirect product of a monoid with a group. Let $M$ be a monoid, $G$ a group, $\Aut{M}$ the automorphism group of $M$, and $\alpha\colon G\to\Aut{M}$ a group homomorphism. For $m\in M$ and $g\in G$ we write $gm$ for the image of $m$ under the automorphism $g\alpha$ (so that we assume that automorphisms act on the left). The \emph{semidirect product} $M\rtimes G$ with respect to $\alpha$ is the set $M\times G$ equipped with the following multiplication: for all $m,m'\in M$, $g,g'\in G$,
\[
(m,g)(m',g'):=\left(m(gm'),gg'\right).
\]
The multiplication is easily seen to be associative so that $M\rtimes G$ is a semigroup.

The following result was first proved by Margolis and Pin~\cite[Propositions~3.6 and 3.7]{Margolis&Pin:1984} by language-theoretical tools. Pin~\cite{Pin:1995} asked for its purely semigroup-theoretic proof. Such a proof was then published by Auinger and Steinberg~\cite{Auinger&Steinberg:2005}.

\begin{proposition}\label{prop:semidirect}
Every semidirect product of a finite $\mathrsfs{J}$-trivial monoid with a finite group divides the power semigroup of some finite group.
\end{proposition}

Now we register a further property of monoids of Hall relations. As mentioned, the monoid $\mathscr{H}_n$ contains both the monoid $\mathscr{R}_n$ and the group $\mathbb{S}_n$. Observe that $\mathbb{S}_n$ acts on $\mathscr{R}_n$ by conjugation since the relation $\pi\rho\pi^{-1}$ is reflexive for every $\rho\in\mathscr{R}_n$ and every $\pi\in\mathbb{S}_n$. This defines a group homomorphism $\mathbb{S}_n\to\Aut\mathscr{R}_n$ that gives rise to the semidirect product $\mathscr{R}_n\rtimes\mathbb{S}_n$.

\begin{proposition}\label{prop:semidirectHn}
The monoid $\mathscr{H}_n$  is a homomorphic image of the semidirect product $\mathscr{R}_n\rtimes\mathbb{S}_n$.
\end{proposition}

\begin{proof}
Define a map $f$ on $\mathscr{R}_n\rtimes\mathbb{S}_n$ by $(\rho,\pi)f:=\rho\pi$ for every pair $(\rho,\pi)\in\mathscr{R}_n\times\mathbb{S}_n$. Since $\rho$ contains the equality relation, the product $\rho\pi$ contains the permutation~$\pi$ whence $\rho\pi$ is a Hall relation. Thus, the map $f$   sends $\mathscr{R}_n\rtimes\mathbb{S}_n$ into $\mathscr{H}_n$.

We aim to show that $f\colon\mathscr{R}_n\rtimes\mathbb{S}_n\to\mathscr{H}_n$ is an onto homomorphism. If $\sigma\in\mathscr{H}_n$ is an arbitrary Hall relation, take a permutation $\tau$ such that $\tau\subseteq\sigma$ and consider the relation $\sigma\tau^{-1}$. Clearly, $\sigma\tau^{-1}$ is reflexive and $\left(\sigma\tau^{-1},\tau\right)f=\sigma\tau^{-1}\tau=\sigma$. Thus, the map $f$ is surjective.

It remains to verify that $f$ is a homomorphism. Taking any $\rho,\rho'\in\mathscr{R}_n$ and any $\pi,\pi'\in\mathbb{S}_n$, we see that
\begin{align*}
\left((\rho,\pi)(\rho',\pi')\right)f&=\left((\rho(\pi\rho'\pi^{-1}),\pi\pi')\right)f&&\text{by definition of semidirect product}\\
                                    &=\rho(\pi\rho'\pi^{-1})\pi\pi'&&\text{by definition of the map $f$}\\
                                    &=\rho\pi\rho'\pi'&&\\
                                    &=\left((\rho,\pi)\right)f\cdot\left(\rho',\pi')\right)f&&\text{by definition of the map $f$.\hspace*{1.18cm}\qed}
\end{align*}
\end{proof}

Now we can easily deduce the ``only if'' of Theorem~\ref{thm:bg=pg} from Theorem~\ref{thm:main}. (The ``if'' part of Theorem~\ref{thm:bg=pg} is immediately ensured by the fact that power semigroups of finite groups are block-groups---see, e.g., \cite[Proposition~2.4]{Pin:1980} for this fact.) Let $S$ be a block-group. Combining Theorem~\ref{thm:main} and Proposition~\ref{prop:semidirectHn}, we see that $S$ divides a semidirect product of a finite $\mathrsfs{J}$-trivial monoid with a finite group, while Proposition~\ref{prop:semidirect} tells us that any such product divides the power semigroup of another finite group. The division relation is transitive, whence $S$ divides the power semigroup of the latter group.

As mentioned, there exists a purely semigroup-theoretic proof of Proposition~\ref{prop:semidirect}. Therefore, a direct algebraic proof of Theorem~\ref{thm:main} would provide a new algebraic proof of Henckell and Rhodes's theorem. Thus, the relationship between our main result and Henckell and Rhodes's theorem is to a large extent parallel to that between Straubing's and Simon's theorems.

\smallskip

We conclude with reminding a longstanding open question concerning Hall relations \cite[Problem~13]{Kim:1982}: what is the cardinality of the monoid $\mathscr{H}_n$?

\end{document}